\documentclass[letterpaper, 10 pt, conference]{ieeeconf}

\IEEEoverridecommandlockouts
\usepackage{tikz}
\usepackage[utf8]{inputenc}
\usepackage{csquotes}
\usepackage[english]{babel}
\usepackage[acronym]{glossaries}
\usepackage{blindtext}
\usepackage{amsmath}
\usepackage{amssymb}

\usepackage{amsthm}
\usepackage{mathtools}
\usepackage{physics}
\usepackage[makeroom]{cancel}

\newtheorem{theorem}{Theorem}[section]

\newtheorem{corollary}{Corollary}[theorem]
\newtheorem{lemma}{Lemma}[section]

\newtheorem{proposition}{Proposition}[section]
\newtheorem*{proposition*}{Proposition}

\newtheorem{assumption}{Assumption}[section]

\theoremstyle{remark}

\newtheorem*{remark*}{Remark}

\theoremstyle{definition}

\DeclareMathOperator{\atantwo}{atan2}
\newcommand{\hessian}[1]{\ensuremath{H_{#1}}}

\newcommand{\kronmatrix}[1]{\ensuremath{\boldsymbol{#1}}}
\newcommand{\vecmatrix}[1]{\ensuremath{\boldsymbol{\vec{#1}}}}

\newglossaryentry{vectorize}{name=\ensuremath{\text{vec}},description={vectorisation of
matrix}}
\newglossaryentry{identity}{name=\ensuremath{I},description={identity matrix}}
\newglossaryentry{ones}{name=\ensuremath{\vec{1}},description={vector of ones}}
\newglossaryentry{zeros}{name=\ensuremath{\vec{0}},description={vector of zeros}}
\newglossaryentry{input}{name=\ensuremath{u},description={high level control input}}

\newglossaryentry{xposition}{name=\ensuremath{a},description={x position}}
\newglossaryentry{yposition}{name=\ensuremath{b},description={y position}}
\newglossaryentry{position}{name=\ensuremath{r},description={position}}
\newglossaryentry{stackedpos}{name=\ensuremath{r}, description={stacked positions of all agents}}
\newglossaryentry{orientation}{name=\ensuremath{\theta},description={orientation}}
\newglossaryentry{vel}{name=\ensuremath{v},description={velocity}}
\newglossaryentry{angvel}{name=\ensuremath{\omega},description={angular velocity}}
\newglossaryentry{angerr}{name=\ensuremath{\phi},description={orientation error}}
\newglossaryentry{poserr}{name=\ensuremath{\xi},description={position error}}
\newglossaryentry{poserrx}{name=\ensuremath{\xi_{\gls{xposition}}},description={x position error}}
\newglossaryentry{poserrxi}{name=\ensuremath{\xi_{\gls{xposition}_i}},description={x position error for i-th coordinate}}
\newglossaryentry{poserrxj}{name=\ensuremath{\xi_{\gls{xposition}_j}},description={x position error j-th coordinate}}
\newglossaryentry{poserry}{name=\ensuremath{\xi_{\gls{yposition}}},description={y position error}}
\newglossaryentry{poserryi}{name=\ensuremath{\xi_{\gls{yposition}_i}},description={y position error for i-th coordinate}}
\newglossaryentry{poserryj}{name=\ensuremath{\xi_{\gls{yposition}_j}},description={y position error for j-th coordinate}}
\newglossaryentry{state}{name=\ensuremath{x},description={full state}}
\newglossaryentry{unobsstate}{name=\ensuremath{w}, description={part of the state not affecting cost function}}
\newglossaryentry{outputmap}{name=\ensuremath{g},description={output map}}
\newglossaryentry{dynamics}{name=\ensuremath{f},description={plant + llc dynamics}}
\newglossaryentry{asymptoticposition}{name=\ensuremath{\tilde{\gls*{position}}},description={asymptotic position}}

\newglossaryentry{llcinput}{name=\ensuremath{u},description={low level controller target}}
\newglossaryentry{llcinputx}{name=\ensuremath{\gls{llcinput}_{\gls{xposition}}},description={low level controller target x position}}
\newglossaryentry{llcinputxi}{name=\ensuremath{\gls{llcinput}_{\gls{xposition}_i}},description={low level controller target x position for i-th coordinate}}
\newglossaryentry{llcinputxj}{name=\ensuremath{\gls{llcinput}_{\gls{xposition}_j}},description={low level controller target x position for j-th coordinate}}
\newglossaryentry{llcinputy}{name=\ensuremath{\gls{llcinput}_{\gls{yposition}}},description={low level controller target y position}}
\newglossaryentry{llcinputyi}{name=\ensuremath{\gls{llcinput}_{\gls{yposition}_i}},description={low level controller target y position for i-th coordinate}}
\newglossaryentry{llcinputyj}{name=\ensuremath{\gls{llcinput}_{\gls{yposition}_j}},description={low level controller target y position for j-th coordinate}}
\newglossaryentry{llcinputgain}{name=\ensuremath{k},description={low level controller gain}}

\newglossaryentry{graph}{name=\ensuremath{G},description={formation graph}}
\newglossaryentry{edges}{name=\ensuremath{\mathcal{E}},description={formation edges}}
\newglossaryentry{vertices}{name=\ensuremath{\mathcal{V}},description={agents}}
\newglossaryentry{incidence}{name=\ensuremath{B},description={incidence matrix}}
\newglossaryentry{distances}{name=\ensuremath{d},description={distances matrix for edges}}
\newglossaryentry{distance}{name=\ensuremath{d},description={directional distance on edge}}
\newglossaryentry{laplacian}{name=\ensuremath{L_{\gls{graph}}},description={Laplacian of undirected graph}}
\newglossaryentry{nagents}{name=\ensuremath{N},description={number of agents}}
\newglossaryentry{neighbours}{name=\ensuremath{\mathcal{N}},description={neighbours of agent in undirected graph}}
\newglossaryentry{neighboursdirected}{name=\ensuremath{\mathcal{N}_{\gls{incidence}}},description={neighbours of agents in directed graph given by incidence matrix}}

\newglossaryentry{target}{name=\ensuremath{\mathbf{\tau}},description={target position}}
\newglossaryentry{targetagentdist}{name=\ensuremath{\delta},description={distance from agents in fomration with respect to the target}}

\newglossaryentry{cost}{name=\ensuremath{\Phi},description={objective function}}
\newglossaryentry{reducedcost}{name=\ensuremath{\tilde{\gls{cost}}},description={objective function}}
\newglossaryentry{inputsteadystatemap}{name=\ensuremath{\hat{h}},description={input to steady state map}}
\newglossaryentry{inputsteadystateoutputmap}{name=\ensuremath{h},description={input to steady state output map}}
\newglossaryentry{inputsteadystateoutputmapandinput}{name=\ensuremath{H},description={input to steady state output map and input}}
\newglossaryentry{feedbackoptimizationepsilon}{name=\ensuremath{\epsilon},description={feedback optimization epsilon}}
\newglossaryentry{jacobian}{name=\ensuremath{\mathbb{J}},description={Jacobian}}

\newglossaryentry{costformationgain}{name=\ensuremath{\gamma_1},description={cost function formation gain}}
\newglossaryentry{costtargetgain}{name=\ensuremath{\gamma_2},description={cost function target tracking gain}}

\newglossaryentry{costlipschitzconst}{name=\ensuremath{l},description={cost function lipschitz const}}
\newglossaryentry{hlipschitzconst}{name=\ensuremath{q},description={input steady state output map lipschitz const}}
\newglossaryentry{llcexpstabilityalpha}{name=\ensuremath{\alpha},description={lower bound llc Lyapunov function}}
\newglossaryentry{llcexpstabilitybeta}{name=\ensuremath{\beta},description={upper bound llc Lyapunov function}}
\newglossaryentry{llcexpstabilitygamma}{name=\ensuremath{\gamma},description={llc Lie derivative coefficient}}
\newglossaryentry{llcexpstabilitymu}{name=\ensuremath{\mu},description={llc partial derivative coefficient}}
\newglossaryentry{anglebound}{name=\ensuremath{\kappa},description={angle bound}}
\usepackage{hyperref}
\usepackage[capitalise]{cleveref}
\let\labelindent\relax
\usepackage{enumitem}

\title{\LARGE \bf
Distributed Feedback Optimisation for Robotic Coordination
}

\author{Antonio Terpin, Sylvain Fricker, Michel Perez, Mathias Hudoba de Badyn, Florian Dörfler
\thanks{The authors are with the Automatic Control Laboratory, ETH Z\"{u}rich, Physikstrasse 3, 8092 Z\"{u}rich, Switzerland. Emails: \texttt{\{aterpin, frickers, miperez, mbadyn, dorfler\}@ethz.ch}. This research is supported by the SNSF through NCCR Automation. 
}
}

\begin{document}

\maketitle
\thispagestyle{plain}
\pagestyle{plain}

\begin{abstract}
Feedback optimisation is an emerging technique aimed at steering a system to an optimal steady state for a given objective function. We show that it is possible to employ this control strategy in a distributed manner. Moreover, we prove asymptotic convergence to the set of optimal configurations. To this scope, we show that exponential stability is needed only for the portion of the state that affects the objective function. This is showcased by driving a swarm of agents towards a target location while maintaining a target formation. Finally, we provide a sufficient condition on the topological structure of the specified formation to guarantee convergence of the swarm in formation around the target location.
\end{abstract}

\section{Introduction}
\emph{Feedback optimisation} is an emerging technique aimed at steering a system to a trajectory computed online that is optimal with respect to a selected objective function, relying on little information on the controlled plant.
In this paper, we apply feedback optimisation on distributed systems and extend convergence results to systems where a part of the state is only asymptotically but not exponentially stable. The results are then used to drive a swarm of robots in a given target formation towards a target location in  a distributed manner.

A thorough review on the feedback-based optimisation methodology can be found in \cite{Hauswirth2021OptimizationControllers}. Remarkably, the plant dynamics need not be known. Instead, this approach relies only on the knowledge of the steady-state input-output sensitivity, allowing model-free optimisation and constraint handling \cite{Hauswirth2020Anti-WindupOptimization, Haberle2021Non-ConvexConstraints}. Furthermore, being a feedback-based approach results in the well-known advantages of feedback control, namely robustness to model mismatch and disturbances \cite{Colombino2019TowardsOptimization, Bernstein2019OnlineOptimization}. A classical example is congestion control for cyber-networks, where source controllers and link dynamics are the modelled feedback loop \cite{Wang2011AOptimization}. 
In recent years, gradient projection algorithms and feedback optimisation gained traction within the field of power systems. The fact that operational constraints are satisfied at all times make these techniques feasible for online implementation \cite{Gan2016AnNetworks, DallAnese2018OptimalPursuit}. A detailed overview on offline and online control techniques for electric power systems can be found in \cite{Molzahn2017ASystems}. Stability of such systems is studied in \cite{Menta2018StabilitySystems} and a recent experimental validation obtained on power grids empirically shows the premises of this approach \cite{Ortmann2020ExperimentalGrids}. 
The applicability of the methodology for linear time-invariant systems with saddle-flow dynamics and constrained convex optimisation problems is shown in  \cite{Chang2019Saddle-FlowOptimization}, and non-smooth dynamical systems arising in time-varying optimisation are addressed in \cite{Tang2018AOptimization, Hauswirth2018Time-varyingSystems}. \cite{Simonetto2020Time-VaryingApplications} reviews a broad class of algorithms for time-varying optimizations and shows how this can be applied to drive a single robot towards a target while avoiding collision. Convergence and stability analysis for regulation of linear time-invariant systems towards the optimal solution of a time-varying convex optimisation problem is studied in \cite{Colombino2020OnlineTracking}. Constraints are included in \cite{Lawrence2020Linear-ConvexControl, Bianchin2021Time-VaryingFlows} and \cite{Hauswirth2020OnFlows} extends to non-linear systems and non-convex problems. 

Much of the work on robotic coordination for flocking relies on classical approaches inspired by the so-called Reynolds principles  \cite{Reynolds1987FlocksModel} and typically employs the concept of potential forces \cite{Olfati-Saber2006FlockingTheory}. Along the lines of closed-loop optimisation for robotics coordination, feedback optimisation is used in \cite{Belgioioso2021Sampled-DataTracking} and \cite{Krilasevic2021LearningControl} to learn generalised Nash Equilibrium in a non-cooperative game-theoretical setting. However, both of these require centralised or semi-decentralised algorithms.
To the best of our knowledge, this paper is the first work that uses feedback optimisation for robotic coordination in a distributed manner. Our theoretical analysis on the asymptotic convergence of the closed-loop system builds instead on \cite{Hauswirth2021TimescaleOptimization}, where the authors quantify the required timescale separation to ensure stability and convergence of the interconnection of an exponentially stable plant and different schemes of feedback optimisation. 

Our contributions are threefold. First, we show that it is possible to implement a feedback optimisation scheme in a distributed manner using robotic coordination as running example. Second, we prove the convergence of the swarm to the configuration that minimises the selected cost function. 
In particular, we build on \cite{Hauswirth2021TimescaleOptimization} to show that exponential stability is required only for the portion of the state that affects the considered cost function whereas for the remainder of the state variables, only asymptotic stability is necessary. Furthermore, we derive conditions on the objective function coefficients and topology of the formation graph. These guarantee that the optimal closed-loop steady-state is such that the agents asymptotically gather in formation around a target location.

The remainder of this paper is organised as follows. \autoref{problemstatement} introduces the problem and the control scheme we want to pursue. In \autoref{mainresult}, we analyse the closed-loop system and we present the main results of our work. 
Finally, \autoref{examples} presents empirical results and concrete instances of the considered problem.

\subsection{Notation}
Given a $n$-tuple $(x_1, x_2, \dots, x_n)$, $x = [x_1, x_2, \dots, x_n]^T = [x_i]^T_{i \in \{1,\dots, n\}}$ is its associated vector and $\textbf{diag}(x^T)$ is the square matrix with the components of $x$ on its diagonal. The 2-norm is denoted $\norm{x}$ and for a matrix $M$, $\norm{M}$ is the norm induced by the 2-norm $\norm{\cdot}$. The spectrum of $M$ is denoted by $\text{spec}[M]$ and its null space is $\text{null}[M]$.

The identity matrix of size $n$ is $\gls*{identity}_n$, whereas $\gls*{ones}_n$ and $\gls*{zeros}_n$ denote the column vector of ones and zeros, respectively. The Kronecker product is denoted by $\otimes$ and we define $\kronmatrix{M}_n = M \otimes \gls*{identity}_n$ and $\vecmatrix{M}_n = \gls*{ones}_n \otimes M$.

Given a function $f\colon\mathbb{R}^n\to \mathbb{R}$, $\nabla f(x) = [\frac{\partial{f(x)}}{\partial{x_i}}]^T_{i \in \{1, \dots, n\}}$ is the gradient of $f$, and with $y = [x_i]^T_{i \in I}$, $I \subseteq \{1, \dots, n\}$, we define $\nabla_y f(x) = [\frac{\partial{f(x)}}{\partial{x_i}}]^T_{i \in I}$. The cardinality of a set is denoted by $\abs{I}$. Finally, the \emph{Jacobian} of $g \colon \mathbb{R}^n \to \mathbb{R}^m$ is denoted by $\gls*{jacobian}g$. We adopt the definition $\gls*{jacobian}g(x) = [\nabla g_i(x)]^T_{i \in \{1,\dots,n\}}$.

\subsection{Graph Theory preliminaries}
An \emph{unweighted undirected graph} is a pair $\gls*{graph} = (\gls*{vertices},\gls*{edges})$, where $\gls*{vertices} = \{1, \dots, n\}$ is the set of nodes and $\gls*{edges} \subseteq S^2(\gls*{vertices})$ is the set of edges, $S^k(\gls*{vertices})$ being the symmetric $k$-th power of $\gls*{vertices}$. The neighbours of a node $i$ are defined as $\gls*{neighbours}(i) = \{j \in \gls*{vertices} | \{i,j\} \in \gls*{edges}\}$. The \emph{adjacency matrix} $\mathcal{A}$ of the graph is defined entry-wise as $[\mathcal{A}]_{ij} \in \{0, 1\}$, $[\mathcal{A}]_{ij} = 1$ if and only if $\{i,j\} \in \gls*{edges}$. The \emph{Laplacian} of the graph is the matrix $\gls*{laplacian} = \textbf{diag}(\gls*{ones}_{n}^T\mathcal{A}^T)- \mathcal{A}$.

It is always possible to assign to each edge $\{i,j\}$ of $\gls*{graph}$ a unique label $e \in \{1, \dots, m\}$, where $m$ is the number of distinct edges in $\gls*{edges}$, and an arbitrary orientation. An edge is then described as an ordered pair $e \equiv (i,j)$. 
The incidence matrix $\gls*{incidence} \in \{-1,0,1\}^{n\times m}$ is used to describe an arbitrary orientation. That is, $[\gls*{incidence}_{ie}] = 1$ if and only if $e \equiv (i,j)$, $[\gls*{incidence}_{ie}] = -1$ if and only if $e \equiv (j,i)$ and finally $[\gls*{incidence}_{ie}] = 0$ if and only if the edge $e$ is not incident to node $i$. 

The arbitrary orientation selected for $\gls*{incidence}$ allows to define in-neighbours of a node $i$ as $\gls*{neighboursdirected}^{i}(i) = \{e \in \gls*{edges} |\exists \, j \in \gls*{vertices}, e \equiv (j,i)\}$. Similarly, the out-neighbours of $i$ are defined as $\gls*{neighboursdirected}^{o}(i)= \{e \in \gls*{edges} |\exists \,  j \in \gls*{vertices}, e \equiv (i,j)\}$.

It is useful to recall that $\gls*{laplacian} = \gls*{incidence}\gls*{incidence}^T$. If $\gls*{graph}$ is connected, then $\text{null}[\gls*{incidence}^T] = \text{null}[\gls*{laplacian}] = \text{span}[\gls*{ones}_{n}]$ \cite[Lemma 6.2, Theorem 6.6]{FrancescoBullo2018LecturesSystems}, where the first equality is a consequence of the Finite Rank Lemma \cite[Theorem 7.6]{Lygeros2015LectureZurich} with linear map $\gls*{incidence}$, adjoint $\gls*{incidence}^T$ and Hilbert spaces $\mathbb{R}^{n}$ and $\mathbb{R}^{m}$ with the canonical inner product. All the remaining $n - 1$ eigenvalues of $\gls*{laplacian}$ are strictly positive \cite[Lemma 6.5]{FrancescoBullo2018LecturesSystems}, with the second-smallest one, denoted by $\lambda_2$, known as the \emph{algebraic connectivity} of the graph \cite[Definition 6.7]{FrancescoBullo2018LecturesSystems}.

\section{Problem Statement}\label{problemstatement}
The goal of this work is to show how to apply feedback optimisation in a distributed manner by driving a swarm of $\gls*{nagents}$ agents into a given \emph{target formation} around a given \emph{target location} $\gls*{target}$. Recall that for feedback optimization in general, we only need to know the steady-state input-output sensitivities. However, for the sake of a more detailed and design-oriented analysis, and to show how to relax some of the assumptions made in \cite{Hauswirth2021TimescaleOptimization}, we consider the plant model to be known. 

We consider unicycle dynamics for the generic $i$-th agent with state $\gls*{state}_i = [\gls*{position}_i^T,\gls*{orientation}_i]^T$, where $\gls*{position}_i = [\gls*{xposition}_i,\gls*{yposition}_i]^T \in \mathbb{R}^2$ is the position of the $i$-th agent in the $\gls*{xposition}-\gls*{yposition}$ plane and $\gls*{orientation}_i \in (-\pi,\pi]$ its orientation with respect to the $\gls*{xposition}$-axis. The dynamics are
\begin{equation}\label{eq:plant}
    \dot{\gls*{xposition}}_i = \gls*{vel}_i \cos(\gls*{orientation}_i),\quad \dot{\gls*{yposition}}_i = \gls*{vel}_i \sin(\gls*{orientation}_i),\quad\dot{\gls*{orientation}}_i = \gls*{angvel}_i,
\end{equation}
where the low-level control inputs $\gls*{vel}_i$ and $\gls*{angvel}_i$ are to be defined.

In particular, we design the actuation mechanism available on each agent to track a given fixed reference position $\gls*{llcinput}_i$. Consider the relative displacement error $\gls*{poserr}_i \in \mathbb{R}^2$  between the agents current position $\gls*{position}_i$ and the fixed reference position $\gls*{llcinput}_i$ and the relative heading error $\gls*{angerr}_i \in (-\pi,\pi]$ denoting the angle between the agents orientation $\gls*{orientation}_i$ and the straight line connecting the agents position to the reference position $\gls*{llcinput}_i$. The error variables then read as
\begin{align*}
    \gls*{poserrxi} &= \gls*{llcinputxi} - \gls*{xposition}_i, & \gls*{poserryi} &= \gls*{llcinputyi} - \gls*{yposition}_i,\\ 
    \gls*{poserr}_i &= \sqrt{\gls*{poserrxi}^2 + \gls*{poserryi}^2},  & \gls*{angerr}_i &= \atantwo(\gls*{poserryi}, \gls*{poserrxi})- \gls*{orientation}_i,
\end{align*}
with resulting error dynamics (see \cref{proofs:expstabilityllc})
\begin{equation}\label{eq:err:dynamics}
    \begin{split}
        \dot{\gls*{poserr}_i} &= -\gls*{vel}_i\cos(\gls*{angerr}_i),\\
        \dot{\gls*{angerr}_i} &= \frac{\gls*{vel}_i}{\gls*{poserr}_i}\sin(\gls*{angerr}_i) - \gls*{angvel}_i.
    \end{split}
\end{equation}
We propose the low-level control (LLC) law ($\gls*{llcinputgain}_i > 0$)%
\begin{equation}\label{eq:plant:llc}
    \begin{split}
        \gls*{vel}_i &= \gls*{llcinputgain}_i\gls*{poserr}_i\cos(\gls*{angerr}_i),\quad \\
        \gls*{angvel}_i &= \gls*{llcinputgain}_i\left(\cos(\gls*{angerr}_i) + 1\right)\sin(\gls*{angerr}_i).
    \end{split}
\end{equation}

\begin{lemma}\label{expstabilityllc}
The low-level control law \eqref{eq:plant:llc} 
almost globally asymptotically stabilises \eqref{eq:err:dynamics} around the origin. 
Moreover, $\gls*{poserr} = 0$ is a globally exponentially stable equilibrium.
\end{lemma}
\begin{proof}
The proof is provided in \cref{proofs:expstabilityllc}.
\end{proof}

Every agent has access to its own global position and the relative positions of its neighbours specified through the unweighted formation graph $\gls*{graph} = (\gls*{vertices},\gls*{edges})$, where the set of vertices $\gls*{vertices}$ consists of all $\gls*{nagents}$ agents and the set of edges $\gls*{edges}$ captures the structure of the target formation. We assume the undirected version of $\gls*{graph}$ to be connected, and the target formation to be uniquely defined by the \emph{desired inter-agent relative displacements} $\gls*{distance}_{ij}$ for all $(i,j) \in \gls*{edges}$. That is, $\forall\, i\neq j \in \gls*{vertices}$, $\gls*{stackedpos}_i - \gls*{stackedpos}_j$ is uniquely defined whenever  $\gls*{stackedpos}_i - \gls*{stackedpos}_j = \gls*{distance}_{ij}\;\forall (i,j) \in \gls*{edges}$, i.e. when the agents are in the target formation.

Recall that the orientation can be specified through the incidence matrix $\gls*{incidence}$ (see example in \autoref{examples}).  

For simplicity, we assume that a robot always has access to the relative displacement to the agents it is adjacent to in the target formation specified by $\gls*{graph}$, regardless of their current distance.

Let $\gls*{distances}=[[\gls*{distance}_{ij}^T]_{(i,j) \in \gls*{edges}}]^T$ be the stacked desired inter-agent relative displacements according to the ordering of the edges given by the incidence matrix $\gls*{incidence}$, and let $\gls*{position} = [[\gls*{position}_i^T]_{i \in \gls*{vertices}}]^T$ be the stacked positions of all the agents. The target location of the swarm, $\gls*{target}$, is assumed to be known for all agents. 
Finally, we define $r^* = [[\gls*{position}_i^{*^T}]_{i \in \gls*{vertices}}]^T = \vecmatrix{\gls*{target}}_{\gls*{nagents}} + \gls*{targetagentdist}$ to be the \emph{desired final configuration}, where $\gls*{targetagentdist}$ represents the displacements of the agents from the target location $\gls*{target}$ when being in the target formation defined by \gls*{distances}, that is $\kronmatrix{\gls*{incidence}}_2^T\gls*{stackedpos}^* = \gls*{distances}$. Moreover, when agents are in the desired final configuration, it holds that $\frac{1}{\gls*{nagents}}\sum_{i \in \gls*{vertices}} r_i^* = \gls*{target}$.

To tackle this problem by means of feedback optimisation, we propose the cost function
\begin{subequations}\label{eq:cost:1}
\begin{align}
    \gls*{cost}(\gls*{stackedpos}) = \frac{\gls*{costformationgain}}{2}\sum_{(i,j) \in \gls*{edges}} \norm{\gls*{position}_i - \gls*{position}_j - \gls*{distance}_{ij}}^2 \label{eq:cost:1:formation}\\
    + \frac{\gls*{costtargetgain}}{2}\sum_{i \in \gls*{vertices}}\norm{\gls*{position}_i - \gls*{target}}^2\label{eq:cost:1:target},
\end{align}
\end{subequations}
where $\gls*{costformationgain},\gls*{costtargetgain} > 0$ denote the weights on the formation error and the distance from the target respectively. 
We manipulate the terms of the cost function to write it in matrix form:
\begin{align*}
    \text{\eqref{eq:cost:1:formation}} &= \gls*{cost}_{\gls*{distance}}(\gls*{stackedpos}) =\frac{\gls*{costformationgain}}{2}\sum_{(i,j) \in \gls*{edges}}(\gls*{position}_i - \gls*{position}_j)^T(\gls*{position}_i - \gls*{position}_j) \\ &\qquad \qquad \qquad -\gls*{costformationgain}\sum_{(i,j) \in \gls*{edges}}(\gls*{position}_i - \gls*{position}_j)^T\gls*{distance}_{ij} + const. \\
    &= \frac{\gls*{costformationgain}}{2} \left[ \gls*{stackedpos}^T(\gls*{laplacian} \otimes \gls*{identity}_2)\gls*{stackedpos} - 2\gls*{distances}^T(B \otimes \gls*{identity}_2)^T\gls*{stackedpos}\right] + const.\\
    \text{\eqref{eq:cost:1:target}} &= \gls*{cost}_{\gls*{target}}(\gls*{stackedpos}) =  \frac{\gls*{costtargetgain}}{2}\sum_{i \in \gls*{vertices}}\left(\gls*{position}_i^T\gls*{position}_i - 2\gls*{target}^T\gls*{position}_i\right) + const.  \\
    &= \frac{\gls*{costtargetgain}}{2}\left[\gls*{stackedpos}^T\gls*{stackedpos} - 2(\gls*{ones}_{\gls*{nagents}} \otimes \gls*{target})^T\gls*{stackedpos}\right] + const.
\end{align*}
Hence, the cost function can equivalently be expressed as
\begin{multline*}
    \gls*{cost}(\gls*{stackedpos}) = \frac{\gls*{costformationgain}}{2} \left[ \gls*{stackedpos}^T\kronmatrix{\gls*{laplacian}}_2\gls*{stackedpos} - 2\gls*{distances}^T\kronmatrix{B}_2^T\gls*{stackedpos}\right]\\
    + \frac{\gls*{costtargetgain}}{2}\left[\gls*{stackedpos}^T\gls*{stackedpos} - 2\vecmatrix{\gls*{target}}_{\gls*{nagents}}^T\gls*{stackedpos}\right] + const.,
\end{multline*}
and its gradient is
\begin{equation}\label{eq:cost:gradient}
    \nabla\gls*{cost}(\gls*{stackedpos}) = \gls*{costformationgain}\left[\kronmatrix{\gls*{laplacian}}_2\gls*{stackedpos} - \kronmatrix{\gls*{incidence}}_2\gls*{distances}\right] +\gls*{costtargetgain}\left[\gls*{stackedpos} - \vecmatrix{\gls*{target}}_{\gls*{nagents}}\right].
\end{equation}
\begin{figure}[!ht]
\centering
  \includegraphics[width =\columnwidth]{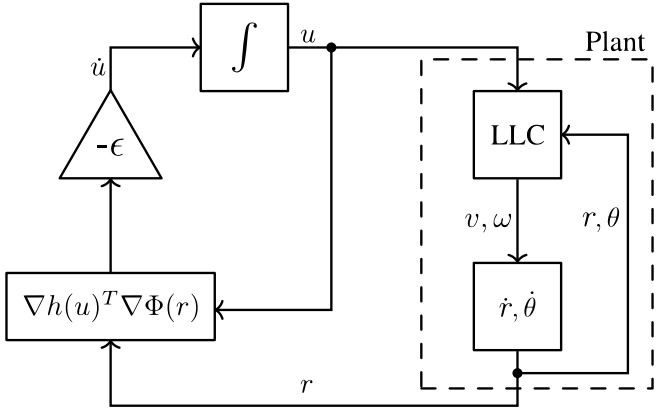}
  \caption{Control scheme for the considered problem.}
  \label{fig:controlproblem}
\end{figure}

\noindent In general, the closed loop system for the feedback optimisation control scheme is
\begin{subequations}\label{eq:closedloopsystem}
\begin{align}
\dot{\gls*{state}} &= f(\gls*{state},\gls*{input})\label{eq:statedynamics}\\
y &= g(\gls*{state},\gls*{input}) =  \gls*{stackedpos}\label{eq:outputmap}\\
\dot{\gls*{input}} &= -\gls*{feedbackoptimizationepsilon}\gls*{jacobian}\gls*{inputsteadystateoutputmap}(u)^T\nabla\gls*{cost}(\gls*{stackedpos})\label{eq:controllaw},
\end{align}
\end{subequations}
where the steady-state input-output map,  $\gls*{inputsteadystateoutputmap}$, is defined as $\gls*{inputsteadystateoutputmap}(\gls*{input}) = \lim_{t \to \infty} \gls*{outputmap}(\gls*{state}(t),\gls*{input})$ and its sensitivity is $\gls*{jacobian}\gls*{inputsteadystateoutputmap}(\gls*{input})$. \eqref{eq:statedynamics} is the plant dynamics comprising the low-level controller, \eqref{eq:outputmap} is the output map and \eqref{eq:controllaw} is the feedback-optimisation control-law dynamics.

For the problem setup outlined above, we have \[\gls*{inputsteadystateoutputmap}(\gls*{input}) = \gls*{input}, \; \gls*{jacobian}\gls*{inputsteadystateoutputmap}(u) = \gls*{identity}_{2\gls*{nagents}},\] 
and \eqref{eq:controllaw} reads as
\begin{equation}\label{eq:controllaw:explicit}
    \dot{\gls*{input}} = -\gls*{feedbackoptimizationepsilon}\gls*{costformationgain}\left[\kronmatrix{\gls*{laplacian}}_2\gls*{stackedpos} - \kronmatrix{\gls*{incidence}}_2\gls*{distances}\right]-\gls*{feedbackoptimizationepsilon}\gls*{costtargetgain}\left[\gls*{stackedpos} - \vecmatrix{\gls*{target}}_{\gls*{nagents}}\right],
\end{equation}
while \eqref{eq:statedynamics} is given by the plant dynamics \eqref{eq:plant} with low level controller \eqref{eq:plant:llc}. The scheme of the closed-loop system is shown in \autoref{fig:controlproblem}.
%

\begin{remark*}
As the focus of this paper is on distributed feedback optimisation, we consider a fixed global frame that is known to all agents for simplicity of exposition. As we shall discuss in \autoref{distributedcontrollaw}, each agent only needs access to relative displacements to its neighbours and to the target. Thus, the information needed is local and any result in this paper can be established also considering different fixed local frames for all agents, which is shortly outlined subsequently. 

The error variables are independent of the reference frame by definition. Therefore, as long as every agent can locate itself in its own fixed local frame, its respective contribution to the cost function is the same as when using a global frame. Moreover, the input dynamics are also independent of the reference frame and the resulting trajectory of every agent can be projected from the global frame to its local frame by homogeneous transformation.
\end{remark*}

\section{Closed Loop Analysis}\label{mainresult}
We now present the main results of this paper. Namely, we first show in \autoref{distributedcontrollaw} that the closed-loop control law~\eqref{eq:closedloopsystem} is distributed. Then, in \autoref{asymptoticbehaviour}, we prove that the robotic swarm asymptotically converges and optimises the cost function \eqref{eq:cost:1}. Finally, in \autoref{optimality} we provide bounds on the gains related to the topological structure of the target formation $\gls*{distances}$ to guarantee that the optimal configuration with respect to the specified cost function corresponds to the agents being in target formation around the target location $\gls*{target}$.

\subsection{Feedback optimisation as a distributed control law}\label{distributedcontrollaw}
To show that the control-law in \eqref{eq:controllaw} uses only local information for the $i$-th component, we derive the input dynamics for the $i$-th agent:
\begin{align*}
    \dot{\gls*{input}_i} &= -\gls*{feedbackoptimizationepsilon}\gls*{costformationgain}\sum_{j \in \gls*{neighboursdirected}^{o}(i)} (\gls*{position}_i - \gls*{position}_j - \gls*{distance}_{ij}) \\
    & \qquad \qquad -\gls*{feedbackoptimizationepsilon}\gls*{costformationgain}\sum_{j \in \gls*{neighboursdirected}^{i}(i)} (\gls*{position}_i - \gls*{position}_j + \gls*{distance}_{ji}) \\
    &\qquad \qquad \qquad \qquad -\gls*{feedbackoptimizationepsilon}\gls*{costtargetgain}(\gls*{position}_i - \gls*{target}).
\end{align*}

It can be observed that the control law depends only on local information  $y_i = [\gls*{position}_i^T, [\gls*{position}_j^T]_{j \in \gls*{neighbours}(i)}]^T$. Furthermore, the information actually used is relative, i.e., each agent needs only access to the relative displacement from the target ($\gls*{position}_i - \gls*{target}$) and from its neighbours ($\gls*{position}_i - \gls*{position}_j$).

\subsection{Asymptotic behaviour}\label{asymptoticbehaviour}
The key motivation of this subsection is that even if both the plant (\ref{eq:plant},\ref{eq:plant:llc}) and the control input dynamics  \eqref{eq:controllaw} are asymptotically stable, there is no guarantee a priori that their interconnection is. To this scope, \cite[Theorem III.2]{Hauswirth2021TimescaleOptimization} requires exponential stability of the entire state of the plant. However, in the following we show that we only require exponential stability for the portion of the state that is actually measured in the output and affects the cost, as long as asymptotic stability of the system is given.

\begin{assumption}\label{expstableobservedstate}
There exists a differentiable function $W(\gls*{stackedpos}, \gls*{llcinput}) \geq 0$ and $\gls*{llcexpstabilityalpha}, \gls*{llcexpstabilitybeta}, \gls*{llcexpstabilitygamma}, \gls*{llcexpstabilitymu} > 0$ such that:
\vspace{3pt}
\begin{enumerate}[label=\roman*.]
    \setlength\itemsep{3pt}
    \item $\gls*{llcexpstabilityalpha}\norm{\gls*{position} - \gls*{inputsteadystateoutputmap}(\gls*{llcinput})}^2 \leq W(\gls*{stackedpos}, \gls*{llcinput}) \leq \gls*{llcexpstabilitybeta}\norm{\gls*{position} - \gls*{inputsteadystateoutputmap}(\gls*{llcinput})}^2$ 
    \item $\nabla_{\gls*{stackedpos}} W(\gls*{position}, \gls*{llcinput})^T\dot{\gls*{position}} \leq -\gls*{llcexpstabilitygamma}\norm{\gls*{position} - \gls*{inputsteadystateoutputmap}(\gls*{llcinput})}^2$  
    \item $\norm{\nabla_{\gls*{llcinput}}W(\gls*{stackedpos}, \gls*{llcinput})} \leq \gls*{llcexpstabilitymu}\norm{\gls*{position} - \gls*{inputsteadystateoutputmap}(\gls*{llcinput})}.$
\end{enumerate}
\end{assumption}
\noindent Further, we need the weakened Lipschitz condition in \cite[Assumption III.1]{Hauswirth2021TimescaleOptimization}, and a Lipschitz condition on $\gls*{inputsteadystateoutputmap}$.
\begin{assumption}[{\cite[Assumption III.1]{Hauswirth2021TimescaleOptimization}}]\label{cost:lipschitz}
There exists $ \gls*{costlipschitzconst} > 0$ such that
\begin{multline*}
    \norm{\gls*{jacobian}\gls*{inputsteadystateoutputmap}(u)^T\left[\nabla\gls*{cost}(\gls*{stackedpos}') - \nabla\gls*{cost}(\gls*{stackedpos})\right]} \leq \gls*{costlipschitzconst}\norm{\gls*{stackedpos}'-\gls*{stackedpos}}, \\
    \forall \gls*{stackedpos}',\gls*{stackedpos}, u \in \mathbb{R}^{2\gls*{nagents}}.
\end{multline*}
\end{assumption}

\begin{assumption}\label{cost:hlipschitz} The steady-state input-output map,$\gls*{inputsteadystateoutputmap}$, is $\gls*{hlipschitzconst}$-Lipschitz continuous and $\text{null}[\gls*{jacobian}\gls*{inputsteadystateoutputmap}(\gls*{input})^T] = \{\mathbf{0}_{2\gls*{nagents}}\}$.
\end{assumption}
We now provide a simple extension to \cite[Lemma A.2]{Hauswirth2021TimescaleOptimization}.
\begin{lemma}\label{compactsublevelsets}
    Consider a stable system \eqref{eq:statedynamics} satisfying \autoref{expstableobservedstate} and \autoref{cost:hlipschitz}. Let $\gls*{outputmap}(\gls*{state})$ be the measured portion of the state $\gls*{state}$ and $Z(\gls*{state},\gls*{input}) = V(\gls*{input}) + W(\gls*{outputmap}(\gls*{state}), \gls*{input})$, where $V$ is a continuous positive-semidefinite function with compact sublevel sets, and $W$ is the function defined in \autoref{expstableobservedstate}. Then, the sublevel sets of $Z$ are compact.
\end{lemma}
\begin{proof}
The proof is provided in \cref{proofs:compactsublevelsets}.
\end{proof}
With this result, we are now ready to state our variation of \cite[Theorem III.2]{Hauswirth2021TimescaleOptimization} that allows us to prove asymptotic convergence for the system~\eqref{eq:closedloopsystem}. 

\begin{theorem}[Asymptotic convergence]\label{asymptoticconvergencetheorem} Suppose that \autoref{expstableobservedstate}, \autoref{cost:lipschitz} and \autoref{cost:hlipschitz} hold, and that the objective function $\gls*{cost}(\gls*{stackedpos}) = \gls*{cost}(\gls*{outputmap}(\gls*{state}))$  is differentiable with compact sub-level sets.
Then, the closed loop system \eqref{eq:closedloopsystem} converges asymptotically to the set of critical points of $\gls*{cost}(\gls*{stackedpos})$ whenever $\gls*{feedbackoptimizationepsilon} < \sqrt{\frac{\gls*{llcexpstabilitygamma}}{\gls*{llcexpstabilitymu}\gls*{costlipschitzconst}}}$.
\end{theorem}
\begin{proof}
We consider a LaSalle's function of the form
\begin{align*}
    \Psi(\gls*{state}, \gls*{input}) =& (1-\delta)\gls*{reducedcost}(\gls*{input}) + \delta W(\gls*{outputmap}(\gls*{state}),\gls*{input}),\\
    \dot{\Psi}(\gls*{state}, \gls*{input}) =& (1-\delta)\nabla\gls*{reducedcost}(\gls*{input})^T\dot{\gls*{input}} \\
    & + \delta \nabla_{\gls*{outputmap}} W(\gls*{outputmap}(\gls*{state}), \gls*{input})^T \gls*{jacobian}\gls*{outputmap}(\gls*{state})\dot{\gls*{state}}\\
    &+ \delta \nabla_{\gls*{input}} W(\gls*{outputmap}(\gls*{state}), \gls*{input})^T\dot{\gls*{input}}.
\end{align*}
where $\gls*{reducedcost}(\gls*{input}) = \gls*{cost}(\gls*{inputsteadystateoutputmap}(\gls*{input}))$.

The second term is bounded via (\autoref{expstableobservedstate}-ii) as,
\begin{equation*}
\nabla_{\gls*{outputmap}} W(\gls*{outputmap}(\gls*{state}), \gls*{input})^T \gls*{jacobian}\gls*{outputmap}(\gls*{state})\dot{\gls*{state}} = \nabla_{\gls*{stackedpos}} W(\gls*{position}, \gls*{input})^T\dot{\gls*{position}} \leq -\gls*{llcexpstabilitygamma}\norm{\gls*{position} - \gls*{inputsteadystateoutputmap}(\gls*{input})}^2,
\end{equation*}
whereas for the first and third term we can proceed as in the proof of \cite[Lemma III.1]{Hauswirth2021TimescaleOptimization}.

Let $\kappa(\gls*{state},\gls*{input}) = -\gls*{jacobian}\gls*{inputsteadystateoutputmap}(u)^T\nabla\gls*{cost}(\gls*{outputmap}(\gls*{state})) = \dot{\gls*{input}} / \gls*{feedbackoptimizationepsilon}$. Then, for the first term we have
\begin{align*}
    \nabla\gls*{reducedcost}(\gls*{input})^T&\kappa(\gls*{state},\gls*{input}) = \nabla\gls*{cost}(\gls*{inputsteadystateoutputmap}(\gls*{input}))^T\gls*{jacobian}\gls*{inputsteadystateoutputmap}(u)\kappa(\gls*{state},\gls*{input})\\
    &= \left(\nabla\gls*{cost}(\gls*{inputsteadystateoutputmap}(\gls*{input})\right)^T-\nabla\gls*{cost}(\gls*{outputmap}(\gls*{state}))^T)\gls*{jacobian}\gls*{inputsteadystateoutputmap}(u)\kappa(\gls*{state},\gls*{input})\\
    &\qquad\qquad\qquad+\nabla\gls*{cost}(\gls*{outputmap}(\gls*{state}))^T\gls*{jacobian}\gls*{inputsteadystateoutputmap}(u)\kappa(\gls*{state},\gls*{input})\\
    &\leq \norm{\gls*{jacobian}\gls*{inputsteadystateoutputmap}(u)^T\left(\nabla\gls*{cost}(\gls*{inputsteadystateoutputmap}(\gls*{input}))-\nabla\gls*{cost}(\gls*{outputmap}(\gls*{state}))\right)}\norm{\kappa(\gls*{state},\gls*{input})}\\
    &\qquad\qquad\qquad+\nabla\gls*{cost}(\gls*{outputmap}(\gls*{state}))^T\gls*{jacobian}\gls*{inputsteadystateoutputmap}(u)\kappa(\gls*{state},\gls*{input})\\
    &\leq \gls*{costlipschitzconst}\norm{\gls*{inputsteadystateoutputmap}(\gls*{input})-\gls*{outputmap}(\gls*{state})}\norm{\kappa(\gls*{state},\gls*{input})}\\
    &\qquad\qquad\qquad+\nabla\gls*{cost}(\gls*{outputmap}(\gls*{state}))^T\gls*{jacobian}\gls*{inputsteadystateoutputmap}(u)\kappa(\gls*{state},\gls*{input}),
\end{align*}
where in the first inequality we use the Cauchy-Schwartz inequality \cite[Theorem 7.1]{Lygeros2015LectureZurich}, and in the second we use \autoref{cost:lipschitz}. Recalling the definition of $\kappa$, we finally obtain 
\begin{multline*}
\nabla\gls*{reducedcost}(\gls*{input})^T\kappa(\gls*{state},\gls*{input})\\
\leq\gls*{costlipschitzconst}\norm{\gls*{outputmap}(\gls*{state})-\gls*{inputsteadystateoutputmap}(\gls*{input})}\norm{\kappa(\gls*{state},\gls*{input})}-\norm{\kappa(\gls*{state},\gls*{input})}^2.
\end{multline*}

\noindent For the third term we have (\autoref{expstableobservedstate}-iii)
\begin{align*}
    \nabla_{\gls*{input}} W(\gls*{outputmap}(\gls*{state}), \gls*{input})^T\kappa(\gls*{state},\gls*{input})
&\leq \norm{\nabla_{\gls*{input}}W(\gls*{outputmap}(\gls*{state}), \gls*{input})}\norm{\kappa(\gls*{state},\gls*{input})}\\
&\leq \gls*{llcexpstabilitymu}\norm{\gls*{stackedpos} - \gls*{inputsteadystateoutputmap}(\gls*{input})}\norm{\kappa(\gls*{state}, \gls*{input})}.
\end{align*}

\noindent Hence, since $\gls*{outputmap}(\gls*{state}) = \gls*{stackedpos}$, we obtain 
\begin{equation}\label{upperboundpsi}
\dot{\Psi}(\gls*{state},\gls*{input}) \leq \begin{bmatrix}\norm{\kappa(\gls*{state},\gls*{input})}\\\norm{\gls*{stackedpos}-\gls*{inputsteadystateoutputmap}(\gls*{input})}\end{bmatrix}^T\Lambda \begin{bmatrix}\norm{\kappa(\gls*{state},\gls*{input})}\\\norm{\gls*{stackedpos}-\gls*{inputsteadystateoutputmap}(\gls*{input})}\end{bmatrix},\end{equation}
$$\Lambda = \begin{bmatrix}
-(1-\delta) & \frac{1}{2}\gls*{feedbackoptimizationepsilon}(\gls*{costlipschitzconst}(1-\delta)+\gls*{llcexpstabilitymu}\delta)\\
\frac{1}{2}\gls*{feedbackoptimizationepsilon}(\gls*{costlipschitzconst}(1-\delta)+\gls*{llcexpstabilitymu}\delta)& -\gls*{llcexpstabilitygamma}\delta
\end{bmatrix},$$
which is negative definite if (\cite[pp.296]{PetarKokotovic19877.Systems})
\begin{equation*}
    \delta = \frac{\gls*{costlipschitzconst}}{\gls*{llcexpstabilitymu} + \gls*{costlipschitzconst}} \text{ and } \gls*{feedbackoptimizationepsilon} < \sqrt{\frac{\gls*{llcexpstabilitygamma}}{\gls*{llcexpstabilitymu}\gls*{costlipschitzconst}}}.
\end{equation*}
Moreover, we notice that the for the left hand side of \eqref{upperboundpsi} to be zero, we need the right hand side to cancel out as well (negative definite quadratic form). This is equivalent to having $\gls*{stackedpos} = \gls*{inputsteadystateoutputmap}(\gls*{input})$ and $\kappa(\gls*{state},\gls*{input}) = 0$.

Since $\dot{\Psi} \leq 0$, we know that the sublevel sets of $\Psi$ are invariant and using \autoref{compactsublevelsets} we conclude that they are also compact. Therefore, taking $P = \{(\gls*{state},\gls*{input}) \in \mathbb{R}^{n+p} \mid \Psi(\gls*{state},\gls*{input}) \leq \Psi(\gls*{state}(t_0), \gls*{input}(t_0))\}$ we have that for any initial condition $\left(\gls*{state}(t_0),\gls*{input}(t_0)\right)$ the trajectories $(\gls*{state}(t), \gls*{input}(t))$ converge to the largest invariant subset $S \subseteq P$ for which $\dot{\Psi} = 0$.

In particular, denoting by $\gls*{unobsstate} \in \mathbb{R}^{t}$, $t < n$, the portion of the state that does not affect $\gls*{cost}$, and assuming without loss of generality $x = [\gls*{inputsteadystateoutputmap}(\gls*{input})^T, \gls*{unobsstate}^T]^T$, we have that
\begin{align*}
    S &= \left\{(\gls*{state},\gls*{input}) \in P \mid \dot{\Psi}(\gls*{state},\gls*{input}) = 0\right\}\\
    &\subseteq \left\{(\gls*{state},\gls*{input}) \in P \mid \norm{\gls*{position} - \gls*{inputsteadystateoutputmap}(\gls*{input})} = 0, \kappa(\gls*{state},\gls*{input}) = 0\right\}\\
    &=\left\{([\gls*{inputsteadystateoutputmap}(\gls*{input})^T, \gls*{unobsstate}^T]^T, \gls*{input}) \in P \mid \gls*{unobsstate} \in \mathbb{R}^t, \nabla\gls*{reducedcost}(\gls*{input}) = 0\right\}\\
    &=\left\{([\gls*{inputsteadystateoutputmap}(\gls*{input})^T, \gls*{unobsstate}^T]^T, \gls*{input}) \in P \mid \gls*{unobsstate} \in \mathbb{R}^t, \nabla\gls*{cost}(\gls*{inputsteadystateoutputmap}(\gls*{input})) = 0\right\},
\end{align*}
where in the second to last step we use $\nabla\gls*{reducedcost}(\gls*{input}) = \kappa(\gls*{inputsteadystateoutputmap}(\gls*{input}), \gls*{input})$ and in the last step we use (\autoref{cost:hlipschitz}),
$$\mathbf{0} = \nabla\gls*{reducedcost}(\gls*{input}) = \gls*{jacobian}\gls*{inputsteadystateoutputmap}(\gls*{input})^T\nabla\gls*{cost}(\gls*{inputsteadystateoutputmap}(\gls*{input})) \iff \nabla\gls*{cost}(\gls*{inputsteadystateoutputmap}(\gls*{input})) = \mathbf{0}_{2\gls*{nagents}}.$$
\end{proof}
\begin{corollary}[Asymptotic convergence]\label{asymptoticconvergence}
The closed-loop system \eqref{eq:closedloopsystem} converges asymptotically to the set of critical points of \eqref{eq:cost:1} whenever $\gls*{feedbackoptimizationepsilon} < \sqrt{ \frac{\gls*{llcexpstabilitygamma}}{\gls*{llcexpstabilitymu}\gls*{costlipschitzconst}}}$.
\end{corollary}
\begin{proof}
To prove this result, we show that the assumptions in  \autoref{asymptoticconvergencetheorem} hold.
First, we notice that by \autoref{expstabilityllc} the positional error dynamics are exponentially stable for a fixed $\gls*{input}$ and thus, we can use the standard converse Lyapunov theorem \cite[Theorem 3.12]{Khalil2002NonlinearSystems} to claim the existence of a Lyapunov function $W$ that satisfies \autoref{expstableobservedstate}.

For \autoref{cost:lipschitz} we consider the bound (directly using $\gls*{jacobian}\gls*{inputsteadystateoutputmap}(\gls*{input}) = \gls*{identity}_{2\gls*{nagents}}$)
\begin{align*}
    \norm{\nabla\gls*{cost}(\gls*{stackedpos}') - \nabla\gls*{cost}(\gls*{stackedpos})} &=\norm{(\gls*{costformationgain}\kronmatrix{\gls*{laplacian}}_2 + \gls*{costtargetgain}\gls*{identity}_{2\gls*{nagents}})(\gls*{stackedpos}'-\gls*{stackedpos})} \\
    &\leq \norm{\gls*{costformationgain}\kronmatrix{\gls*{laplacian}}_2 + \gls*{costtargetgain}\gls*{identity}_{2\gls*{nagents}}}\norm{\gls*{stackedpos}'-\gls*{stackedpos}},
\end{align*}
so that we can set $\gls*{costlipschitzconst} = \norm{\gls*{costformationgain}\kronmatrix{\gls*{laplacian}}_2 + \gls*{costtargetgain}\gls*{identity}_{2\gls*{nagents}}}$.

Recalling that $\gls*{inputsteadystateoutputmap}(\gls*{input}) = \gls*{input}$, \autoref{cost:hlipschitz} is trivially satisfied. Finally, our cost function \eqref{eq:cost:1} is continuously differentiable with bounded sublevel-sets, because \eqref{eq:cost:1:target} is a positive definite quadratic form centered in 
$\vecmatrix{\gls*{target}}_{\gls*{nagents}}$ and \eqref{eq:cost:1:formation} is non-negative. Since the pre-image of continuous maps of closed sets (for any c-sublevel set, $[0,c]$ is closed) is closed as well, we can conclude that $\gls*{cost}(r)$ has compact sublevel sets.
\end{proof}
\begin{proposition}\label{eigenvalues}
Consider $M \in \mathbb{R}^{n\times n}$. Then \cite{Plemmons1988MatrixJohnson},
\begin{enumerate}[label=\roman*.]\setlength\itemsep{3pt}
    \item $\text{spec}[\alpha M + \beta] = \{\alpha \lambda_1 + \beta, \dots, \alpha \lambda_n + \beta\}$; and
    \item $\exists M^{-1}$, $\text{spec}[M^{-1}] = \{\lambda_1^{-1}, \dots, \lambda_n^{-1}\}$.
\end{enumerate}
\end{proposition}
\begin{corollary}[Optimality]\label{corollary:optimality}
The closed-loop system \eqref{eq:closedloopsystem} minimises the cost function \eqref{eq:cost:1} whenever $\gls*{feedbackoptimizationepsilon} < \sqrt{ \frac{\gls*{llcexpstabilitygamma}}{\gls*{llcexpstabilitymu}\gls*{costlipschitzconst}}}$.
\end{corollary}
\begin{proof}
To assess the strict convexity of \eqref{eq:cost:1} we investigate the second order condition. We derive from \eqref{eq:cost:gradient} $\hessian{\gls*{cost}} = \gls*{costformationgain}(\gls*{laplacian} \otimes \gls*{identity}_2)+ \gls*{costtargetgain}\gls*{identity}_{2\gls*{nagents}}$. Using \autoref{eigenvalues} we have that $\lambda \in \text{spec}[\gls*{costformationgain}\gls*{laplacian} + \gls*{costtargetgain}\gls*{identity}_{\gls*{nagents}}]$ if and only if $\exists \mu \in \text{spec}[\gls*{laplacian}]$ s.t. $(\lambda - \gls*{costtargetgain})/\gls*{costformationgain} = \mu \geq 0$ (\cite[Lemma 6.5]{FrancescoBullo2018LecturesSystems}) and thus, $\lambda \geq \gls*{costtargetgain} > 0$. Moreover, by simple permutation transformation we have  $\text{spec}\left[\gls*{laplacian} \otimes \gls*{identity}_2\right] = \text{spec}\left[ \gls*{identity}_2 \otimes \gls*{laplacian}\right] =  \text{spec}\left[\gls*{laplacian}\right] \cup \text{spec}\left[\gls*{laplacian}\right]$. Therefore, $\hessian{\gls*{cost}} \succ 0$ and the critical point $\gls*{input}^*$ s.t. $\nabla\gls*{cost}(\gls*{inputsteadystateoutputmap}(\gls*{input}^*)) = 0$ is unique and attains the minimum of the cost function \eqref{eq:cost:1}. From \autoref{asymptoticconvergence} the claim follows.
\end{proof}
The result of \autoref{corollary:optimality} does not imply that the closed-loop system \eqref{eq:closedloopsystem} has a unique equilibrium point $(\gls*{state}^*, \gls*{input}^*)$. However, it does imply that the set of equilibrium points share the same locations for the agents. Namely, $\forall (\gls*{state}_1, \gls*{input}_1), (\gls*{state}_2, \gls*{input}_2) \in S$, we have $\gls*{input}_1 = \gls*{input}_2$ and $\gls*{outputmap}(\gls*{state}_1) = \gls*{outputmap}(\gls*{state}_2)$, but in general $\gls*{orientation}_1 \neq \gls*{orientation}_2$. This is not unexpected: different initial conditions might lead to different final orientations for the agents.
\subsection{Topological Considerations on the Optimal Configuration}\label{optimality}
In this subsection, we further investigate the relation between the asymptotic configuration of the swarm and the topology of the formation. Indeed, although \autoref{corollary:optimality} guarantees that the robotics swarm converges to the optimal configuration with respect to \eqref{eq:cost:1}, it is not clear a priori whether this corresponds to the desired final configuration $r^*$. This is exemplified in \autoref{fig:Eformationwrong} in \autoref{examples}, where choosing the cost function gains inappropriately leads to a misshaped final configuration $\gls*{asymptoticposition} = \lim_{t \to \infty} \gls*{stackedpos}(t)$. \begin{lemma}\label{connectivityandformation}
Consider the same setting of \autoref{asymptoticconvergence}. If additionally $\gls*{costformationgain}\lambda_2 \gg \gls*{costtargetgain}$ holds, the final configuration $\gls*{asymptoticposition}$ is approximately in the target formation, namely 
$$\lim_{\gls*{costformationgain}\lambda_2/\gls*{costtargetgain} \to \infty}\kronmatrix{\gls*{incidence}}_2^T\gls*{asymptoticposition} = \gls*{distances}.$$
\end{lemma}
\begin{proof} Using $0 = \nabla\gls*{cost}(\gls*{asymptoticposition})$ (\autoref{asymptoticconvergence}) and $(\gls*{costformationgain}\kronmatrix{\gls*{laplacian}}_2 + \gls*{costtargetgain}\gls*{identity}_{2\gls*{nagents}}) \succ 0$ (\autoref{corollary:optimality}) we have that the final configuration of the agents is $$\gls*{asymptoticposition} = (\gls*{costformationgain}\kronmatrix{\gls*{laplacian}}_2 + \gls*{costtargetgain}\gls*{identity}_{2\gls*{nagents}})^{-1}(\gls*{costformationgain}\kronmatrix{\gls*{incidence}}_2\gls*{distances} + \gls*{costtargetgain}\vecmatrix{\gls*{target}}_{\gls*{nagents}}).$$ 
Using $\kronmatrix{\gls*{incidence}}_2\gls*{distances} = \kronmatrix{\gls*{incidence}}_2\kronmatrix{\gls*{incidence}}_2^T\gls*{stackedpos}^* = \kronmatrix{\gls*{laplacian}}_2\gls*{stackedpos}^*$ and $\vecmatrix{\gls*{target}}_{\gls*{nagents}} = \gls*{stackedpos}^* - \gls*{targetagentdist}$ we obtain 
\begin{align*}
    \gls*{asymptoticposition} &= (\gls*{costformationgain}\kronmatrix{\gls*{laplacian}}_2 + \gls*{costtargetgain}\gls*{identity}_{2\gls*{nagents}})^{-1}\left[(\gls*{costformationgain}\kronmatrix{\gls*{laplacian}}_2+\gls*{costtargetgain}\gls*{identity}_{2\gls*{nagents}})\gls*{stackedpos}^* - \gls*{costtargetgain}\gls*{targetagentdist}\right]\\
    &= \gls*{stackedpos}^* -(\gls*{costformationgain}\kronmatrix{\gls*{laplacian}}_2 + \gls*{costtargetgain}\gls*{identity}_{2\gls*{nagents}})^{-1} \gls*{costtargetgain}\gls*{targetagentdist}.
\end{align*}
Since $\kronmatrix{\gls*{laplacian}}_2$ is symmetric, we consider the decomposition $\kronmatrix{\gls*{laplacian}}_2 = U\textbf{diag}([0, \lambda_2, \dots, \lambda_{\gls*{nagents}}])U^T$.

Then (\autoref{eigenvalues}),
\begin{multline*}
    (\gls*{costformationgain}\kronmatrix{\gls*{laplacian}}_2 + \gls*{costtargetgain}\gls*{identity}_{2\gls*{nagents}})^{-1}\gls*{costtargetgain}\\
    = U\textbf{diag}([1,\frac{\gls*{costtargetgain}}{\gls*{costformationgain}\lambda_2 + \gls*{costtargetgain}},\dots,\frac{\gls*{costtargetgain}}{\gls*{costformationgain}\lambda_{\gls*{nagents}} + \gls*{costtargetgain}}])U^T
\end{multline*}
and taking the limit,
\begin{equation*}
    \lim_{\gls*{costformationgain}\lambda_2 / \gls*{costtargetgain} \to \infty} (\gls*{costformationgain}\gls*{laplacian} + \gls*{costtargetgain}\gls*{identity}_{2\gls*{nagents}})^{-1}\gls*{costtargetgain} = U\textbf{diag}([1, \gls*{zeros}_{\gls*{nagents}-1}^T])U^T,
\end{equation*}
because $0 < \lambda_2 \leq \dots \leq \lambda_{\gls*{nagents}}$ and thus, $\gls*{costformationgain}\lambda_{i} \gg \gls*{costtargetgain}\,\forall\,i\in\gls*{vertices}$.
Hence, for $\gls*{costformationgain}\lambda_2 \gg \gls*{costtargetgain}$, we have
\begin{equation*}
    (\gls*{costformationgain}\kronmatrix{\gls*{laplacian}}_2 + \gls*{costtargetgain}\gls*{identity}_{2\gls*{nagents}})^{-1}\gls*{costtargetgain} \approx U\textbf{diag}([1, \gls*{zeros}_{\gls*{nagents}-1}^T])U^T.
\end{equation*}
Therefore, 
\begin{align*}
    \kronmatrix{\gls*{incidence}}_2^T\gls*{asymptoticposition} &\approx \kronmatrix{\gls*{incidence}}_2^T(\gls*{stackedpos}^* - U\textbf{diag}([1, \gls*{zeros}_{\gls*{nagents}-1}^T])U^T\gls*{targetagentdist})\\
    &=\kronmatrix{\gls*{incidence}}_2^T\gls*{stackedpos}^* - \kronmatrix{\gls*{incidence}}_2^TU\textbf{diag}([1, \gls*{zeros}_{\gls*{nagents}-1}^T])U^T\gls*{targetagentdist} = \gls*{distances}
\end{align*}
and the final configuration is approximately in the target formation.

In the last equation we use the fact that $\kronmatrix{\gls*{incidence}}_2^TU\textbf{diag}([1, \gls*{zeros}_{\gls*{nagents}-1}^T])U^T = 0$. This can be readily seen as
\begin{align*}
    \kronmatrix{\gls*{laplacian}}_2&U\textbf{diag}([1, \gls*{zeros}_{\gls*{nagents}-1}^T])U^T\\
    &= U\textbf{diag}([0, \lambda_2, \dots, \lambda_{\gls*{nagents}}])U^TU\textbf{diag}([1, \gls*{zeros}_{\gls*{nagents}-1}^T])U^T = 0.
\end{align*}
Now, let $H = U\textbf{diag}([1, \gls*{zeros}_{\gls*{nagents}-1}^T])U^T$ and let $H_i$ be the $i$-th column of $H$.
Then, $\forall i \in \gls*{vertices}$, we have that $\gls*{laplacian}H_i = 0$. That is, $H_i \in \text{null}[\kronmatrix{\gls*{laplacian}}_2] = \text{null}[\kronmatrix{\gls*{incidence}}_2^T]$.

Therefore, $0 = \gls*{incidence}^TH = \gls*{incidence}^TU\textbf{diag}([1, \gls*{zeros}_{\gls*{nagents}-1}^T])U^T$.
\end{proof}

\autoref{connectivityandformation} has a topological interpretation. $\lambda_2$ is known as \textit{algebraic connectivity} \cite[Definition 6.7]{FrancescoBullo2018LecturesSystems} and characterises the connectivity of the graph \cite[Lemma 6.9]{FrancescoBullo2018LecturesSystems}. Hence, a stronger connectivity of the specified target formation is expected to result in a final configuration that is approximately the desired one. Moreover, since $\gls*{costformationgain}$ and $\gls*{costtargetgain}$ are gains in the gradient-flow \eqref{eq:controllaw:explicit}, they influence the speed of convergence to the target formation and the target location. Therefore, a stronger connectivity of the formation graph allows for a larger $\gls*{costtargetgain}$ that in turn allows a faster convergence to the target location.

Finally, we show that the agents gather in the target formation around the target location. Notice that we cannot have both a non-trivial target formation and all agents in the same position. Hence, the correctness of the asymptotic behaviour has to be investigated by means of bounds on the distance of the robotic swarm from the target location. 

\begin{theorem}[Correct final configuration]\label{desiredfinalconfiguration}
Consider the same settings of \autoref{connectivityandformation}, and let $\gls*{asymptoticposition}$ be the configuration the agents converge to. Then $\norm{\gls*{asymptoticposition} - \vecmatrix{\gls*{target}}_{\gls*{nagents}}} < \norm{\gls*{targetagentdist}}$, with $\kronmatrix{\gls*{incidence}}_2^T\gls*{asymptoticposition} \approx \gls*{distances}$.
\end{theorem}
\begin{proof}
From \autoref{connectivityandformation} we have $\gls*{asymptoticposition} = \gls*{stackedpos}^* - (\gls*{costformationgain}\kronmatrix{\gls*{laplacian}}_2 + \gls*{costtargetgain}\gls*{identity}_{2\gls*{nagents}})^{-1}\gls*{costtargetgain}\gls*{targetagentdist}$ and thus, 
\begin{align*}
    \norm{\gls*{asymptoticposition} - \vecmatrix{\gls*{target}}_{\gls*{nagents}}} &= \norm{\gls*{asymptoticposition} - \gls*{stackedpos}^* + \gls*{targetagentdist}} \\
    &= \norm{[\gls*{identity}_{2\gls*{nagents}}-(\gls*{costformationgain}\kronmatrix{\gls*{laplacian}}_2 + \gls*{costtargetgain}\gls*{identity}_{2\gls*{nagents}})^{-1}\gls*{costtargetgain}]\gls*{targetagentdist}}\\
    &\leq\norm{\gls*{identity}_{2\gls*{nagents}}-(\gls*{costformationgain}\kronmatrix{\gls*{laplacian}}_2 + \gls*{costtargetgain}\gls*{identity}_{2\gls*{nagents}})^{-1}\gls*{costtargetgain}}\norm{\gls*{targetagentdist}} < \norm{\gls*{targetagentdist}}.
\end{align*}
Indeed, (\autoref{eigenvalues}, \cite[Lemma 6.5]{FrancescoBullo2018LecturesSystems}), $\forall \mu \in \text{spec}[\gls*{identity}_{2\gls*{nagents}}-(\gls*{costformationgain}\kronmatrix{\gls*{laplacian}}_2 + \gls*{costtargetgain}\gls*{identity}_{\gls*{nagents}})^{-1}\gls*{costtargetgain}]$, there exists $\lambda \in \text{spec}[\kronmatrix{\gls*{laplacian}}_2]$ such that $$\mu = 1 - \frac{\gls*{costtargetgain}}{\gls*{costformationgain}\lambda + \gls*{costtargetgain}} = \frac{\gls*{costformationgain}\lambda}{\gls*{costformationgain}\lambda + \gls*{costtargetgain}} \in [0,1),$$
and thus, $\norm{\gls*{identity}_{2\gls*{nagents}}-(\gls*{costformationgain}\kronmatrix{\gls*{laplacian}}_2 + \gls*{costtargetgain}\gls*{identity}_{2\gls*{nagents}})^{-1}\gls*{costtargetgain}} < 1$.
Finally, $\kronmatrix{\gls*{incidence}}_2^T\gls*{asymptoticposition} \approx \gls*{distances}$ is immediate from \autoref{connectivityandformation}.
\end{proof}

As one would intuitively expect, the formation shrinks around the target location due to the attractive force. However, in the limit $\gls*{costformationgain}\lambda_2 \gg \gls*{costtargetgain}$ we achieve the correct inter-agent displacements.

\section{Examples And Empirical Results}\label{examples}
In this section, we provide simulation\footnote{The \textsc{MatLab} code for the simulations is available at \url{github.com/antonioterpin/feedback-optimization-swarm-robotics}} examples to underline the transient behaviour under the provided cost function \eqref{eq:cost:1}, and we show what can go wrong when the assumptions of \autoref{desiredfinalconfiguration} do not hold.
\begin{figure}
    \centering
    \includegraphics[width=\linewidth]{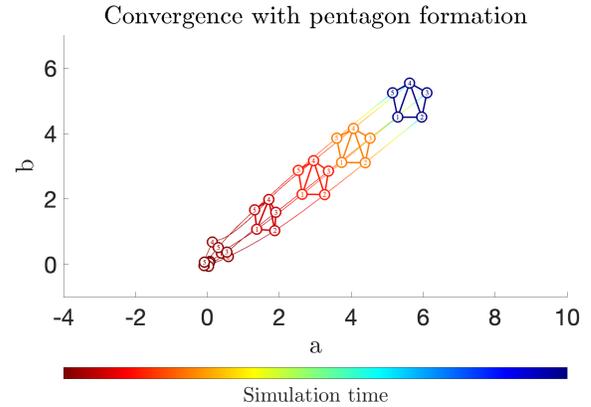}
    \caption{Evolution over time of a swarm of $5$ agents in the $\gls*{xposition}-\gls*{yposition}$ plane. The target formation is a pentagon and the target location is $\gls*{target} = [5.65, 5.03]^T$. The time evolution is colour-coded from red (beginning) to blue (end).}
    \label{fig:pentagonformation}
\end{figure}
In \autoref{fig:pentagonformation} a simulation for a swarm of $5$ agents is shown. The pentagon formation is specified by 
\begin{equation*}
    \gls*{incidence} = 
        \begin{bmatrix}
            1 & 0 & 0 & 0 & -1 & 1 & 0\\
            -1 & 1 & 0 & 0 & 0 & 0 & -1\\
            0 & -1 & 1 & 0 & 0 & 0 & 0\\
            0 & 0 & -1 & 1 & 0 & -1 & 1\\
            0 & 0 & 0 & -1 & 1 & 0 & 0
        \end{bmatrix}
\end{equation*}
and the desired relative inter-agent displacement vector (reshaped in a $2\times\gls*{nagents}$ matrix for readibility reasons) is
\begin{equation*}
    \gls*{distance} = 
        \begin{bmatrix}
        -1.0 & -0.3 & 0.8 & 0.8 & -0.3 & -0.5 & -0.5\\
         0.0 & -1.0 & -0.3 & 0.3 & 1.0 & -1.3 & 1.3
    \end{bmatrix}.
\end{equation*}
The agents start from random positions in a neighbourhood around the origin, they assemble in the target formation and move together towards the specified target location. In \autoref{fig:pentagonformation}, one can appreciate the time evolution visualised by the change in colour. It is apparent that the specified cost function guarantees that the agents do not only drive to the target location before assembling in the target formation, but they do so early on. Hence, the desired transient behaviour is empirically obtained. Characterising it quantitatively is still an open question and a topic of ongoing research.
\begin{figure}
    \centering
    \includegraphics[width=\linewidth]{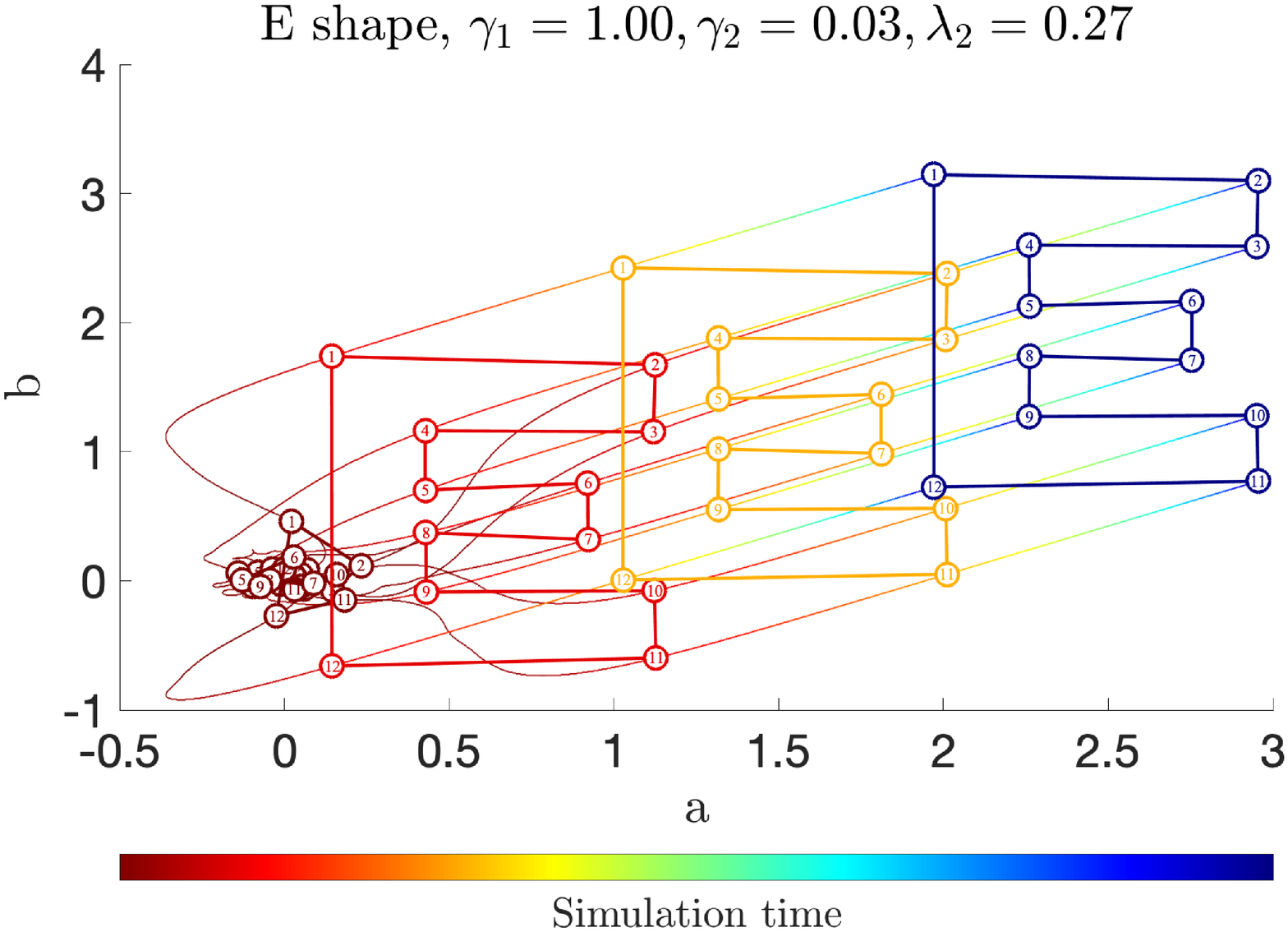}
    \caption{Evolution over time of a swarm of $12$ agents in the $\gls*{xposition}-\gls*{yposition}$ plane. The robots converge to the target formation and location asymptotically, since the conditions of \autoref{desiredfinalconfiguration} hold. The time evolution is colour-coded from red (beginning) to blue (end).}
    \label{fig:Eformation}
\end{figure}
\begin{figure}
    \centering
    \includegraphics[width=\linewidth]{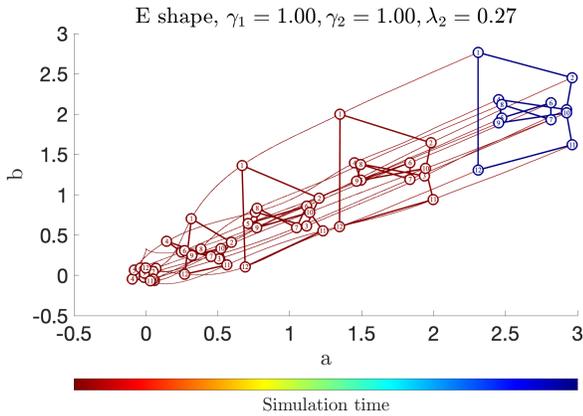}
    \caption{Evolution over time of a swarm of $12$ agents in the $\gls*{xposition}-\gls*{yposition}$ plane. Oppositely to \autoref{fig:Eformation}, the conditions of \autoref{desiredfinalconfiguration} do not hold and thus, the final configuration is not the desired one. That is, the cost function does not capture the desired goal of the task. The time evolution is colour-coded from red (beginning) to blue (end).}
    \label{fig:Eformationwrong}
\end{figure}

Next, we consider an E-shaped formation, which is clearly not circularly symmetric around its centroid and thus, the cost term \eqref{eq:cost:1:target} is expected to shrink and distort the formation for poorly chosen gains.
In \autoref{fig:Eformation} and \autoref{fig:Eformationwrong} the simulations of the robotics swarm with different values for the gains $\gls*{costformationgain}$ and $\gls*{costtargetgain}$ are shown. It is apparent that when the conditions of \autoref{desiredfinalconfiguration} are not fulfilled, the final configuration of the agents is not the desired one (\autoref{fig:Eformationwrong}). On the other hand, we empirically notice that a factor of ${\sim}30$ between $\gls*{costformationgain}\lambda_2$ and $\gls*{costtargetgain}$ is enough to obtain the desired result in \autoref{fig:Eformation}. Moreover, the intuition on the speed of the flocking is empirically verified. Indeed, for $\gls*{costformationgain}\lambda_2 \gg \gls*{costtargetgain}$  (\autoref{fig:Eformation}), the agents converge faster to the target formation and then relatively slowly towards the target location, whereas for $\gls*{costformationgain}\lambda_2 \cancel{\gg} \gls*{costtargetgain}$ (\autoref{fig:Eformationwrong}) we have a fast flocking towards the target location, but the agents do not gather in the target formation.

\section{Conclusions And Future Work}\label{conclusion}
In this article, we applied feedback optimisation to drive a swarm of agents in formation towards a target location in a distributed manner. The correctness of the algorithm was rigorously proved by means of sufficient conditions on the topology of the specified formation and on the gains of the gradient-flow.

To conclude, we want to outline some possible research directions. First, being able to prove the optimality of the final configuration in presence of non-convexities due to obstacle avoidance terms or visibility constraints would represent a natural extension to the work presented in this paper.
Additionally, another interesting venue of research would be to quantify to what extent feedback optimisation implicitly deals with noisy dynamics and input saturation. Finally, quantitatively characterising the transient behaviour of the closed loop system is still a major question to address in the context of feedback optimisation.


\bibliographystyle{ieeetr}
\bibliography{references.bib}
\appendix
\subsection{Proof of \autoref{expstabilityllc}}\label{proofs:expstabilityllc}
In the following we drop the indices for convenience.
First, we observe that, for a fixed $\gls*{llcinput}$, we have
\begin{align}
    \begin{split}
    \dot{\gls*{poserr}} &= \frac{-(\gls*{llcinput} - \gls*{position})^T}{\norm{\gls*{llcinput} - \gls*{position}}}\dot{\gls*{position}} = -\frac{\gls*{poserrx}\gls*{vel}\cos(\gls*{orientation}) + \gls*{poserry}\gls*{vel}\sin(\gls*{orientation})}{\gls*{poserr}} \\
    &= -\gls*{vel}\cos(\gls*{angerr})\label{eq:poserr:dynamics}
    \end{split}\\
    \dot{\gls*{angerr}} &= \frac{\gls*{vel}}{\gls*{poserr}}\sin(\gls*{angerr}) - \gls*{angvel}\label{eq:angerr:dynamics},
\end{align}
where we use trigonometric identities and we notice that $\cos(\atantwo(b,a)) = a / \sqrt{a^2+b^2}$ and $\sin(\atantwo(b,a)) = b / \sqrt{a^2+b^2}$.
Substituting the proposed control law we obtain
\begin{align}
    \dot{\gls*{poserr}} &= -\gls*{llcinputgain}\gls*{poserr}\cos(\gls*{angerr})^2, \label{poserr:dynamics}\\
    \dot{\gls*{angerr}} &= -\gls*{llcinputgain}\sin(\gls*{angerr}). \label{angerr:dynamics}
\end{align}
Noticing that $\gls*{angerr}$ is decoupled from $\gls*{poserr}$, we consider the Lyapunov function 
$V(\gls*{angerr}) = \frac{1}{2} \gls*{angerr}^2,$
which has compact sub-level sets (positive quadratic form). It holds that
\begin{equation*}
    \dot{V}(\gls*{angerr}) = -\gls*{llcinputgain}\gls*{angerr}\sin(\gls*{angerr})\leq 0 \text{ and }
    \dot{V}(\gls*{angerr}) = 0 \iff \sin(\gls*{angerr}) = 0.
\end{equation*}
Moreover, $\gls*{angerr} = \pi$ is an unstable equilibrium of \eqref{angerr:dynamics} and thus, $\gls*{angerr} = 0$ is almost globally asymptotically stable. Namely, any trajectory originating in $(-\pi, \pi)$ will converge to $\gls*{angerr} = 0$ \cite[Theorem 15.4]{FrancescoBullo2018LecturesSystems}.
Furthermore, \eqref{poserr:dynamics} has the explicit solution $$\gls*{poserr}(t) = \gls*{poserr}(t_0)\text{exp}\left\{-\gls*{llcinputgain}\int_{t_0}^{t}\cos(\gls*{angerr}(s))^2ds\right\}.$$
Since $\gls*{angerr} = 0$ is almost globally asymptotically stable, $\exists \; T$ s.t. $\forall\;t > T\;,\abs{\gls*{angerr}(t)} < \eta < \pi/2$ for some $\eta$. Hence,
for $t_0 \leq t \leq T$, we have that (mean value Theorem) $$\exists \; K_t > 0 \text{ s.t. } \gls*{poserr}(t) = \gls*{poserr}(t_0)\text{exp}\left\{-\gls*{llcinputgain}K_t(t-t_0)\right\}.$$
From the continuity of the definite integral there exists a continuous map $t \mapsto K_t$ and thus, $\exists \; K_1 = \gls*{llcinputgain}\min_{t \in [t_0,T]} K_t > 0$.
Whereas, for $t > T$, let $K_T = \int_{t_0}^{T}\cos(\gls*{angerr}(s))^2ds$ and $K_2 / \gls*{llcinputgain} = \cos(\eta)^2$. Then, we have
\begin{align*}
    \int_{t_0}^{t}\cos(\gls*{angerr}(s))^2&ds = \int_{t_0}^{T}\cos(\gls*{angerr}(s))^2ds + \int_{T}^{t}\cos(\gls*{angerr}(s))^2ds\\
    &\geq K_T + \int_{T}^{t}\cos(\eta)^2ds = K_T + \frac{K_2}{\gls*{llcinputgain}} (t - T)\\
    &=\underbrace{K_T + \frac{K_2}{\gls*{llcinputgain}} (t_0 - T)}_{M_2 / \gls*{llcinputgain}} + \frac{K_2}{\gls*{llcinputgain}} (t - t_0)\\
    &= \frac{M_2}{\gls*{llcinputgain}} + \frac{K_2}{\gls*{llcinputgain}} (t - t_0).
\end{align*}

Therefore, with $M = \gls*{poserr}(t_0)\max(1, \text{exp}(-M_2))$ and $0 < \mu = \min(K_1, K_2)$ we have that
$$\gls*{poserr}(t) \leq M\text{exp}\left\{-\mu(t-t_0)\right\}.\qed$$
\subsection{Steady state input output map}\label{steadystatemaps}
From the definition of the steady-state input-output map $\gls*{inputsteadystateoutputmap}$ of the $i$-agent, we set $\gls*{dynamics}([\gls*{inputsteadystateoutputmap}_i(\gls*{input}_i)^T, \gls*{orientation}_i]^T, \gls*{input}_i) = 0$.
Hence,
\begin{align*}
\begin{cases}
    0 = \gls*{vel}_i\cos(\gls*{orientation}_i)\\
    0 = \gls*{vel}_i\sin(\gls*{orientation}_i)
\end{cases} &\iff 0 = \gls*{vel}_i = \norm{\gls*{input}_i - \gls*{inputsteadystateoutputmap}_i(\gls*{input}_i)}
\end{align*}
That is, $\gls*{input}_i = \gls*{inputsteadystateoutputmap}_i(\gls*{input}_i)$. Given \autoref{expstabilityllc} 
 and stacking the maps of the single agents we obtain $\gls*{inputsteadystateoutputmap}(\gls*{input}) = \gls*{input}$.\qed
\subsection{Proof of \autoref{compactsublevelsets}}\label{proofs:compactsublevelsets}
Let $\Omega_c = \left\{(\gls*{state},\gls*{input}) | Z(\gls*{state},\gls*{input}) \leq c\right\}$ be a sublevel set of $Z$. $W \geq 0$ implies $V \leq c$, and since V has compact sublevel sets, $\norm{u} \leq U$ for some $U$. Similarly, $W \leq c$ since $V \geq 0$. Let $\gls*{stackedpos} = \gls*{outputmap}(\gls*{state})$ be the measured part of the state, and $\gls*{unobsstate}$ the unmeasured, stable portion. Notice that the stability of $\gls*{unobsstate}$ implies that for any initial condition $\gls*{state}(t_0)$ there exists $B$ such that $\norm{\gls*{unobsstate}} \leq B$. From \autoref{expstableobservedstate}-i we have $\norm{\gls*{stackedpos} - \gls*{inputsteadystateoutputmap}(\gls*{input})} \leq \frac{1}{\gls*{llcexpstabilityalpha}}W(\gls*{state},\gls*{input}) \leq \frac{c}{\gls*{llcexpstabilityalpha}}$,
and thus,
\begin{align*}
    \norm{\gls*{state}} &\leq \norm{\gls*{unobsstate}} + \norm{\gls*{stackedpos}}\\
    &\leq \norm{\gls*{unobsstate}} + \norm{\gls*{stackedpos} - \gls*{inputsteadystateoutputmap}(\gls*{input})} + \norm{\gls*{inputsteadystateoutputmap}(\gls*{input}) - \gls*{inputsteadystateoutputmap}(0)} + \norm{\gls*{inputsteadystateoutputmap}(0)}\\
    &\leq \norm{\gls*{unobsstate}} + \frac{c}{\gls*{llcexpstabilityalpha}} + \gls*{hlipschitzconst}\norm{\gls*{input}} + \norm{\gls*{inputsteadystateoutputmap}(0)}\\
    &\leq B + \frac{c}{\gls*{llcexpstabilityalpha}} + \gls*{hlipschitzconst}U + \norm{\gls*{inputsteadystateoutputmap}(0)},
\end{align*}
where $\gls*{hlipschitzconst}$ is the Lipschitz constant of $\gls*{inputsteadystateoutputmap}$ (\autoref{cost:hlipschitz}). Therefore, we can conclude that $\norm{\gls*{state}}$ is bounded and thus, $\norm{(\gls*{state}, \gls*{input})}$ is bounded for all $(\gls*{state},\gls*{input}) \in \Omega_c$. The continuity of $V$ and $W$ (and thus, $Z$) allows us to conclude that $\Omega_c$ is also closed. Indeed, the pre-image of a continuous map on a closed set (such as $[0, c]$) is closed.\qed
\end{document}